\documentclass[11pt,a4paper]{article}
\usepackage{amsmath}
\usepackage{amssymb}
\usepackage{amsfonts}
\usepackage{amsthm}
\usepackage{relsize}
\usepackage{setspace}
\usepackage{geometry}
\usepackage{enumerate}
\usepackage{url}

\newtheorem{thm}{Theorem}[section]
\newtheorem{lem}[thm]{Lemma}
\newtheorem{prop}[thm]{Proposition}
\newtheorem{cor}[thm]{Corollary}

\theoremstyle{definition}
\newtheorem{defn}[thm]{Definition}

\newtheorem{rem}[thm]{Remark}

\newcommand{\tdlc}{t.d.l.c$.$}

\newcommand{\Aut}{\mathrm{Aut}}

\newcommand{\F}{\mathrm{F}}
\newcommand{\M}{\mathrm{M}}
\newcommand{\N}{\mathrm{N}}
\renewcommand{\O}{\mathrm{O}}

\newcommand{\Out}{\mathrm{Out}}
\newcommand{\Inn}{\mathrm{Inn}}

\newcommand{\Sym}{\mathrm{Sym}}

\newcommand{\Comm}{\mathrm{Comm}}

\newcommand{\Sylp}{\mathrm{Syl}_p}
\newcommand{\Syls}{\mathrm{Syl}_*}

\newcommand{\piP}{p \in \mathbb{P}}

\newcommand{\bN}{\mathbb{N}}
\newcommand{\bP}{\mathbb{P}}
\newcommand{\bQ}{\mathbb{Q}}

\newcommand{\mcC}{\mathcal{C}}
\newcommand{\mcD}{\mathcal{D}}

\newcommand{\mcL}{\mathcal{L}}
\newcommand{\mcM}{\mathcal{M}}

\newcommand{\mcP}{\mathcal{P}}

\newcommand{\mcS}{\mathcal{S}}
\newcommand{\mcT}{\mathcal{T}}
\newcommand{\mcU}{\mathcal{U}}

\begin{document}

\title{Local Sylow theory of totally disconnected, locally compact groups}

\author{Colin D. Reid\\
Universit\'{e} catholique de Louvain\\
Institut de Recherche en Math\'{e}matiques et Physique (IRMP)\\
Chemin du Cyclotron 2, 1348 Louvain-la-Neuve\\
Belgium\\
colin@reidit.net}

\maketitle

\abstract{We define a local Sylow subgroup of a totally disconnected, locally compact (\tdlc{}) group $G$ to be a maximal pro-$p$ subgroup of an open compact subgroup of $G$.  We use these subgroups to define the $p$-localisation of $G$, a locally virtually pro-$p$ group which maps continuously and injectively to $G$ with dense image, and describe the relationship between the scale and modular functions of $G$ and those of its $p$-localisation.  In the case of locally virtually prosoluble groups, we consider all primes simultaneously using local Sylow bases.}

\section{Introduction}

The usual definition of Sylow subgroups for an arbitrary group is that they are maximal $p$-subgroups.  However, a much better-behaved notion in the class of profinite groups is that of maximal pro-$p$ subgroups; it is for these subgroups that one obtains a generalisation of Sylow's theorem, including its corollary that a (pro-)finite group is generated by its Sylow subgroups (provided one allows topological generation).  It is therefore natural to consider a corresponding `local' theory in totally disconnected, locally compact (\tdlc{}) groups, based on the maximal pro-$p$ subgroups of open compact subgroups.  The aim of this paper is to introduce this notion and to prove some basic results.

Throughout this paper, a property is said to hold `locally' in a \tdlc{} group $G$ if it holds for all open compact subgroups of $G$.

\begin{defn}Let $G$ be a \tdlc{} group and let $K$ be a compact subgroup of $G$.  We define the \emph{localisation} $G_{(K)}$ of $G$ at $K$ to be the topological group formed by endowing the commensurator $\Comm_G(K)$ of $G$ in $K$ with the coarsest group topology such that $K$ is open (which is also the finest group topology such that $K$ is compact).  Note that the inclusion map $G_{(K)} \rightarrow G$ is continuous.  Given a prime number $p$, define the \emph{$p$-localisation} $G_{(p)}$ to be $G_{(S)}$, where $S$ is a local $p$-Sylow subgroup, that is, a maximal pro-$p$ subgroup of an open compact subgroup of $G$.\end{defn}

The following result shows that we do not need to be concerned with the choice of $S$ in the definition of $p$-localisation, and that the $p$-localisation incorporates a large part of $G$ as a whole.

\begin{thm}\label{commdense}Let $G$ be a \tdlc{} group, let $p$ be a prime and let $S$ be a local $p$-Sylow subgroup of $G$.  Then $\Comm_G(S)$ is a dense subgroup of $G$, and the $G$-conjugacy class of $\Comm_G(S)$ and the isomorphism type of $G_{(p)}$ as a topological group do not depend on the choice of $S$.  There is an injective homomorphism from $\Out(G)$ to $\Out(G_{(p)})$ (where we only allow automorphisms that respect the topology).\end{thm}

Some properties of $G_{(p)}$ and the embedding $G_{(p)} \rightarrow G$ are closely related to questions about the extent to which a $p$-Sylow subgroup of an open compact subgroup $U$ fails to be normal in $U$.

\begin{defn}Let $U$ be a profinite group.  Define the \emph{$p$-core} $\O_p(U)$ to be the unique largest normal pro-$p$ subgroup of $U$.  Say $U$ is \emph{$p$-normal} if $\O_p(U)$ is the $p$-Sylow subgroup of $U$.\end{defn}

\begin{thm}\label{pnormthm}Let $G$ be a \tdlc{} group, let $p$ be a prime and let $S$ be a local $p$-Sylow subgroup of $G$.
\begin{enumerate}[(i)]
\item Suppose $G_{(p)}$ has an open compact normal subgroup.  Then $G$ is locally virtually $p$-normal.
\item Let $H$ be a \tdlc{} group that is locally virtually $p$-normal.  Let $\phi: G \rightarrow H$ be a continuous homomorphism.  Then $\phi(\Comm_G(S)) = \phi(G)$.
\item The commensurator of $S$ contains the $p$-core of every open compact subgroup of $G$.
\item Suppose that $G$ is first-countable and that there is an open compact subgroup $U$ of $G$ that has only finitely many maximal open normal subgroups.  Then the following are equivalent:
\begin{enumerate}[(a)]
\item $G$ is locally virtually $p$-normal;
\item $\Comm_G(S)=G$;
\item $\Comm_G(S)$ contains a dense normal subgroup of $G$.\end{enumerate}\end{enumerate}\end{thm}

For part (i), note that if $G$ is profinite, then $G_{(p)}$ is topologically residually finite.  As a result, $p$-localisations of profinite groups can be examples of \tdlc{} groups in which the open normal subgroups have trivial intersection, but do not form a base of neighbourhoods of the identity.    It is also worth noting that the $p$-cores of open compact subgroups of a \tdlc{} group form a single commensurability class (see Proposition \ref{corecomm}).

\

It is a feature of \tdlc{} groups that the modular function takes only rational values, that is, the scaling of the Haar measure by conjugation (or indeed by any topological automorphism) is always by a rational factor.  In a locally virtually pro-$p$ group, the values are always powers of $p$.  We will use a subscript $p$ to denote the $p$-part of a rational number, or the set of $p$-parts of a set of rational numbers.  The modular function of a \tdlc{} group and that of its $p$-localisation are related in a straightforward way.

\begin{prop}\label{goodmod}Let $G$ be a \tdlc{} group, let $S$ be a local $p$-Sylow subgroup of $G$ and let $x \in \Comm_G(S)$.  Let $\Delta$ be the modular function on $G$, and let $\Delta_{(p)}$ be the modular function on $G_{(p)}$.  Then $\Delta(x)_p = \Delta_{(p)}(x)$.  As a consequence, $\Delta(G)_p = \Delta_{(p)}(G_{(p)})$.\end{prop}

A more powerful invariant in the theory of \tdlc{} groups is the scale (function) introduced by G. A. Willis (\cite{Wil1}), a continuous function from a \tdlc{} group (or its automorphism group) to the natural numbers.  Similar considerations apply in terms of divisibility by powers of $p$ as with the modular function.  However, the connection between the scale on $G$ and the scale on $G_{(p)}$ is more complicated.

\begin{thm}\label{locscale}Let $G$ be a \tdlc{} group and let $S$ be a local $p$-Sylow subgroup of $G$.  Write $s$ for the scale on $G$, and $s_{(p)}$ for the scale on $G_{(p)} := G_{(S)}$.  Let $x \in G$.
\begin{enumerate}[(i)]
\item If $x$ commensurates $S$, then $s(x)_p \le s_{(p)}(x)$.
\item Let $U$ be a tidy subgroup for $x$ in $G$.  Then there exists $u \in U$ such that $xu$ commensurates $S$ and such that
\[ s(x)=s(xu); \; s(x)_p = s_{(p)}(xu).\]\end{enumerate}\end{thm}

\begin{cor}\label{scalesub}Let $G$ be a \tdlc{} group.  Then $s(G)_p \subseteq s_{(p)}(G_{(p)})$.  In particular, if $G_{(p)}$ is uniscalar then $s(x)$ is coprime to $p$ for every $x \in G$.\end{cor}

In \S\ref{egsec} we will consider groups $G$ acting on trees that are universal with respect to a prescribed action at each vertex, in the sense of Burger and Mozes, and show how to calculate scale functions for such groups.  These examples will show that if every element of $G$ has scale coprime to $p$, it does not follow that $G_{(p)}$ is uniscalar.

\

We will also consider the class of locally virtually prosoluble groups.  All of the results stated so far have obvious alternative forms in which $G$ is required to be locally virtually prosoluble, but $p$ may be replaced by an arbitrary set of primes $\pi$; the proofs are essentially the same, using Hall's theorem for prosoluble groups instead of Sylow's theorem for profinite groups.  (We leave it as an exercise for the interested reader to state and prove these results; they will not be needed for the purposes of this paper.)  A more interesting approach is to consider Sylow subgroups for every prime at once using local Sylow bases.  We will obtain results analogous to parts of Theorems \ref{commdense} and \ref{pnormthm} in which the local $p$-Sylow subgroup $S$ is replaced by a local Sylow basis, that is, a set of pro-$p$ groups, one for each prime $p$, that are pairwise permutable and whose product is an open prosoluble subgroup of $G$.

\section{Localisation}

We recall the definitions of tidy subgroups and the scale.  Results of Willis can be used to establish a connection between the tidy subgroups of a \tdlc{} group and those of its localisations.

\begin{defn}\label{tidydef}Let $G$ be a \tdlc{} group, let $H$ be a closed normal subgroup of $G$, let $x \in G$, and let $U$ be an open compact subgroup of $H$.  Define the following subgroups of $H$:
\[ U_+ = \bigcap_{i \ge 0}x^iUx^{-i}; \; U_- = \bigcap_{i \le 0} x^iUx^{-i}; \; U_{++} = \bigcup_{j \ge 0}\bigcap_{i \ge j}  x^iUx^{-i}; \; U_{--} = \bigcup_{j \le 0}\bigcap_{i \le j}x^iUx^{-i}.\]

Say that $U$ is \emph{tidy} for $x$ (in $H$) if the following conditions are satisfied:

(I) $U = U_+U_-$;

(II) $U_{++}$ and $U_{--}$ are closed subgroups of $H$.

The \emph{scale} $s_H(x)$ of $x$ on $H$ is the minimum value of $|xUx^{-1}: xUx^{-1} \cap U|$ as $U$ ranges over the open compact subgroups of $H$.  In $\cite{Wil2}$, it is shown that the minimum value of $|xUx^{-1}: xUx^{-1} \cap U|$ is achieved if and only if $U$ is tidy for $x$; in particular this implies that every element has a tidy subgroup.\end{defn}

\begin{lem}Let $G$ be a \tdlc{} group, and let $K$ be a closed compact subgroup of $G$.  Let $x \in \Comm_G(K)$.  Let $U$ be a tidy subgroup for $x$ in $G$.  Then $U$ contains a tidy subgroup for $x$ in $G_{(K)}$.\end{lem}

\begin{proof}Let $T$ be an open compact subgroup of $G_{(K)}$ that is contained in $U$.  By (\cite{Wil1}, Lemma 1), there is an open subgroup of $T$ that satisfies condition (I) for $x$, so we may assume $T$ satisfies condition (I).  Now define
\[ \mcM_U = \{g \in G \mid \exists k \; \forall |n| > k: x^ngx^{-n} \in U \};\]
\[ \mcM_T = \{g \in G_{(K)} \mid \exists k \; \forall |n| > k: x^ngx^{-n} \in T \}.\]
Note that $\mcM_T$ is a subset of $\mcM_U$.  By (\cite{Wil1}, Corollary to Lemma 3), the fact that $U$ is tidy for $x$ in $G$ implies that $\mcM_U \subseteq U$, so $\mcM_T \subseteq U$.  Note also that $U \cap G_{(K)}$ is open and hence closed in $G_{(K)}$.  In particular, $U$ contains the closure $M$ of $\mcM_T$ in $G_{(K)}$.  There is now a procedure, given in (\cite{Wil2}, \S3), giving a tidy subgroup $W$ for $x$ in $G_{(K)}$ that is contained in $TM$ and hence in $U$.\end{proof}

We now give some general conditions on a family of subgroups, analogous to some well-known properties of Sylow subgroups in finite group theory, which will have useful consequences for commensurators and localisation.

\begin{defn}Let $G$ be a group, and let $\Sigma_1, \Sigma_2$ be sets of subgroups.  A \emph{commensuration of $\Sigma_1$ to $\Sigma_2$} is a bijection $\beta: \Sigma_1 \rightarrow \Sigma_2$ such that $\beta(S)$ is commensurate to $S$ for all $S \in \Sigma_1$ and $\beta(S)=S$ for all but finitely many $S \in \Sigma_1$; say $\Sigma_2$ is \emph{commensurate} to $\Sigma_1$ if a commensuration from $\Sigma_1$ to $\Sigma_2$ exists, and write $\Sigma_1 \le_o \Sigma_2$ if in addition $\beta(S)$ contains $S$ as an open subgroup for all $S \in \Sigma_1$.  Given a group $H$ acting on $G$, write $\Comm_H(\Sigma_1)$ for the group of elements $h \in H$ for which $S \mapsto h(S)$ is a commensuration from $\Sigma_1$ to $h(\Sigma_1)$.  Say $h \in H$ \emph{normalises} $\Sigma_1$ if $h(S) = S$ for all $S \in \Sigma_1$.

Given a subset $X$ of $G$, define $\Sigma_1 \cap X := \{S \cap U \mid S \in \Sigma_1\}$.\end{defn}

We do not attempt to define a localisation of the topology `at $\Sigma$' in the cases where $\Sigma$ can consist of more than one subgroup.

\begin{defn}Let $\mcC$ be a class of \tdlc{} groups that is closed under taking open subgroups.  A \emph{locally conjugate structure} $\mcL$ defined for $\mcC$ is an assignment to every $G \in \mcC$ a non-empty set $\mcL(G)$ of (sets of) compact subgroups of $G$ such that:
\begin{enumerate}[(i)]
\item $\mcL$ is compatible with isomorphisms;
\item Given $\Sigma_1,\Sigma_2 \in \mcL(G)$, there is some $g \in G$ such that $g\Sigma_1g^{-1}$ is commensurate to $\Sigma_2$;
\item Given an open subgroup $U$ of $G$, then every element of $\mcL(U)$ is commensurate to an element of $\mcL(G)$ and vice versa.\end{enumerate}\end{defn}

\begin{lem}\label{lcslem}Let $\mcC$ be a class of \tdlc{} groups that is closed under taking open subgroups and let $\mcD$ be the class of \tdlc{} groups that have an open subgroup in $\mcC$.  Let $\mcL$ be a locally conjugate structure defined on $\mcC$.  Then there is a locally conjugate structure $\mcL^*$ defined on $\mcD$ by setting $\mcL^*(G) = \bigcup_{U \in \mcU(G)} \mcL(U)$, where $\mcU(G)$ is the set of open $\mcC$-subgroups of $G$.\end{lem}

\begin{prop}\label{lcsdense}Let $\mcC$ be a class of \tdlc{} groups closed under open subgroups, let $\mcL$ be a locally conjugate structure defined for $\mcC$, let $G \in \mcC$.  Let $U,V \leq_o G$, let $\Sigma_1 \in \mcL(U)$ and let $\Sigma_2 \in \mcL(V)$.
\begin{enumerate}[(i)]
\item Given any open subgroup $W$ of $G$, there exists $w \in W$ such that $w\Sigma_1w^{-1}$ is commensurate to $\Sigma_2$, and thus $w\Comm_G(\Sigma_1)w^{-1} = \Comm_G(\Sigma_2)$.  In particular, if in addition $\Sigma_1$ and $\Sigma_2$ are subgroups of $G$ then $G_{(\Sigma_1)} \cong G_{(\Sigma_2)}$.
\item $\Comm_G(\Sigma_1)$ is a dense subgroup of $G$.\end{enumerate}\end{prop}

\begin{proof}Lemma \ref{lcslem} is clear from the definition.

(i) Without loss of generality $W \le U \cap V$.  By definition, there is some $\Sigma^*_1 \in \mcL(W)$ commensurate to $\Sigma_1$ and $\Sigma^*_2 \in \mcL(W)$ commensurate to $\Sigma_2$; in turn there is $w \in W$ such that $w\Sigma^*_1w^{-1}$ is commensurate to $\Sigma^*_2$.  Clearly $w\Sigma_1w^{-1}$ is commensurate to $w\Sigma^*_1w^{-1}$, so it follows that $w\Sigma_1w^{-1}$ is commensurate to $\Sigma_2$.

(ii) Let $W$ be an open compact subgroup of $G$ and let $x \in G$.  We have $x\Sigma_1x^{-1} \in \mcL(xUx^{-1})$, so by (i) there is $w$ in $W$ such that $w\Sigma_1w^{-1}$ is commensurate to $x\Sigma_1x^{-1}$, so $x \in W\Comm_G(\Sigma_1)$.  As this applies for every open subgroup $W$, it follows that $\Comm_G(\Sigma_1)$ is dense in $G$.\end{proof}

Some results of Willis on properties of the scale on passing to closed subgroups have consequences for localisation.

\begin{lem}[Willis: see \cite{Wil2}, \S4]\label{wilsclem}Let $G$ be a \tdlc{} group, let $H$ be a closed subgroup of $G$ and let $x \in G$ such that $xHx^{-1}=H$.  Let $s_G$ be the scale on $G$ and let $s_H$ be the scale on $H$.
\begin{enumerate}[(i)]
\item We have $s_H(x) \le s_G(x)$.
\item If $\N_G(H)$ is open in $G$ then $s_H(x)$ divides $s_G(x)$.
\item If $H$ is open in $G$ then $s_H(x) = s_G(x)$, in other words, $H$ contains a tidy subgroup for $x$.\end{enumerate}\end{lem}

\begin{prop}\label{sgpscale}Let $G$ be a \tdlc{} group, let $H$ be a closed subgroup of $G$ and let $U$ be an open compact subgroup of $G$; write $K=H \cap U$.  Then $H$ has the same topology as a subgroup of $G_{(K)}$ as it does as a subgroup of $G$; moreover $H$ is an open subgroup of $G_{(K)}$.  Let $s_G$ be the scale on $G$ and let $s_{(K)}$ be the scale on $G_{(K)}$.  Then for $x \in \N_G(H)$, we have $s_{(K)}(x) \le s_G(x)$.  If $\N_G(H)$ is open in $G$ then $s_{(K)}(x)$ divides $s_G(x)$.\end{prop}

\begin{proof}By construction $K$ is an open compact subgroup of $H$ under the topology induced in $G$, and so all open compact subgroups of $H$ are commensurate to $K$.  As the same holds for $H$ under the topology induced in $G_{(K)}$, we conclude that the topologies agree on $H$; in addition, $H$ is open in $G_{(K)}$ as it contains the open subgroup $K$.

Let $s_H$ be the scale defined $H$.  Then given $x \in \N_G(H)$, we can define $s_H(x)$ by considering $x$ as an automorphism of $H$.  Moreover, $\N_G(H)$ preserves the set of open compact subgroups of $H$, and therefore is a subgroup of $G_{(K)}$, so $s_{(K)}(x)$ is also defined on $\N_G(H)$; since $H$ is open in $G_{(K)}$ with compatible topology, it is clear that $s_{(K)}(x) = s_H(x)$.  If $\N_G(H)$ is open in $G$, then $s_{\N_G(H)}(x) = s_G(x)$, so we may assume $G = \N_G(H)$.  The conclusions now follow from Lemma \ref{wilsclem}.\end{proof}

A recent result of P-E. Caprace and N. Monod gives a significant restriction on localisations of profinite groups, or more generally of \tdlc{} residually discrete groups, that is, those \tdlc{} groups whose open normal subgroups have trivial intersection.

\begin{prop}[\cite{CM} Corollary 4.1]\label{capmon}Let $G$ be a residually discrete locally compact group that is compactly generated.  Then the open normal subgroups of $G$ form a base of neighbourhoods of the identity in $G$.\end{prop}

\begin{cor}\label{normprop}Let $G$ be a residually discrete locally compact group and let $K$ be a compact subgroup of $G$ (possibly open).  Let $X$ be a subset of $G$ contained in $\Comm_G(K)$, such that $X$ is contained in a union of finitely many cosets of $K$.  (Equivalently, let $X$ be a relatively compact subset of $G_{(K)}$.)  Then $X$ normalises an open subgroup of $K$.  In particular, $G_{(K)}$ is uniscalar.\end{cor}

\begin{proof}As the topology of $G_{(K)}$ is a refinement of the induced topology, we see that $G_{(K)}$ is residually discrete.  We may therefore assume that $G = G_{(K)}$ and that $K$ is open in $G$.  Consider $H = \langle K,\overline{X}\rangle$ with the induced topology.  Then $H$ is residually discrete (as this is inherited from $G$) and compactly generated, so by Proposition \ref{capmon} there is an open subgroup $V$ of $K$ that is normalised by $H$.\end{proof}

\section{Local Sylow theory}

\begin{defn}A \emph{supernatural number} is a formal product $\prod_{\piP} p^{n_p}$ of prime powers, where each $n_p$ is either a non-negative integer or $\infty$.  Multiplication of supernatural numbers is defined in the obvious manner, giving rise to a semigroup; note that any set of supernatural numbers has a supernatural least common multiple, and by unique factorisation, the supernatural numbers contain a copy of the natural numbers.  Given a supernatural number $x = \prod_{\piP} p^{n_p}$, the \emph{$p$-part} $x_p$ of $x$ is $p^{n_p}$; we will similarly talk about the $p$-part of a set $X$ of supernatural numbers, given by $X_p = \{x_p \mid x \in X\}$.  Define the \emph{index} $|G:H|$ of $H$ in $G$ to be the least common multiple of $|G/N:HN/N|$, regarded as a supernatural number, as $N$ ranges over all open normal subgroups of $G$.  Say $H$ is a \emph{$p$-Sylow subgroup} of $G$ if $H$ is a pro-$p$ group such that $|G:H|_p=1$.

Let $G$ be a \tdlc{} group.  A \emph{local $p$-Sylow subgroup} of $G$ is a $p$-Sylow subgroup of an open compact subgroup of $G$.  Let $\Sylp(G)$ be the set of local $p$-Sylow subgroups of $G$.  We will write $G_{(p)}$ for $G_{(S)}$ where $S \in \Sylp(G)$.\end{defn}

\begin{lem}\label{sylplcs}The assignment $\Sylp$ is a locally conjugate structure on all \tdlc{} groups.\end{lem}

\begin{proof}Sylow's theorem generalises as follows to profinite groups (see for instance \cite{RZ}): given a profinite group $G$, then the $p$-Sylow subgroups of $G$ form exactly one conjugacy class, and every pro-$p$ subgroup of $G$ is contained in a $p$-Sylow subgroup.

Given an open subgroup $H$ of $G$ and a $p$-Sylow subgroup $T$ of $H$, then $T$ lies in some $p$-Sylow subgroup $S$ of $G$; we see that $|S:T| = |G:H|_p$, so $|S:T|$ is finite.  Conversely, given any $p$-Sylow subgroup $S$ of $G$, then $S \cap K$ is a $p$-Sylow subgroup of $K$, where $K$ is the core of $H$ in $G$, and so $S$ is commensurate to a $p$-Sylow subgroup of $H$.  Thus setting $\mcL(G)$ to be the set of $p$-Sylow subgroups of a profinite group $G$ gives a locally conjugate structure on profinite groups.  It follows from Lemma \ref{lcslem} that $\Sylp$ is a locally conjugate structure on all \tdlc{} groups.\end{proof}

\begin{proof}[Proof of Theorem \ref{commdense}]Given Lemma \ref{sylplcs} and Proposition \ref{lcsdense}, all assertions are immediate, except for the existence of the embedding $\theta: \Out(G) \rightarrow \Out(G_{(p)})$.  

One can regard $\Out(G)$ as $H/K$, where $H$ is the holomorph $H = G \rtimes \Aut(G)$ of $G$ (with $\Aut(G)$ equipped with the discrete topology) and $K = G \rtimes \Inn(G)$.  Then $S$ is a local $p$-Sylow subgroup of $H$, so $\Comm_H(S)$ is dense in $H$.  As a result, we see that given $\omega \in \Out(G)$, there is a representative $\alpha \in \Aut(G)$ that commensurates $S$.  Let $\alpha'$ be the induced automorphism on $G_{(p)}$ obtained by restricting $\alpha$ and then refining the topology.  Set $\theta(\omega) = \alpha'\Inn(G_{(p)})$.  Given two possible choices $\alpha_1$ and $\alpha_2$ of $\alpha$, then $\alpha_1\alpha^{-1}_2$ is an inner automorphism of $G$ that commensurates $S$, so induces an inner automorphism of $\Comm_G(S)$ and hence of $G_{(p)}$; thus $\theta(\omega)$ does not depend on the choice of $\alpha$.  Given $\omega_1, \omega_2 \in \Out(G)$, if $\alpha_i \in \Aut(G)$ represents $\omega_i$ and commensurates $S$ for $i=1,2$, then $\alpha_1\alpha^{-1}_2$ represents $\omega_1\omega^{-1}_2$ and commensurates $S$, so $\theta$ is a homomorphism.  Finally, if $\theta(\omega)=1$, then there is a representative $\alpha$ of $\omega$ inducing an inner automorphism $\alpha'$ of $G_{(p)}$, say conjugation by some $x \in G_{(p)}$.  Let $\beta$ be conjugation by $x$ as an automorphism of $G$.  Then $\beta^{-1}\alpha$ is an inner automorphism of $G$ that fixes $\Comm_G(S)$ pointwise; since $\Comm_G(S)$ is dense and centralisers in a topological group are closed, it follows that $\beta = \alpha$, so $\omega$ is the trivial element of $\Out(G)$.\end{proof}

\begin{defn}The \emph{Mel'nikov subgroup} $\M(U)$ of $U$ is the intersection of all maximal open normal subgroups of $U$.  Note that $U/\M(U)$ is finite if and only if $U$ has only finitely many maximal open normal subgroups (that is, open normal subgroups $K$ such that $U/K$ is simple).\end{defn}

\begin{proof}[Proof of Theorem \ref{pnormthm}](i) Let $T$ be an open compact normal subgroup of $G$.  Then $T$ is virtually pro-$p$; by replacing $T$ with $\O_p(T)$, we may assume $T$ is pro-$p$.  Then $T \unlhd G$, since $T$ is closed in $G$ and has dense normaliser.  It follows that for any open compact subgroup $U$ of $G$, we have $U \cap T \le \O_p(U)$.  However, $U \cap T$ is commensurate to $T$ and $T$ is commensurate to $S$, which is in turn commensurate to a $p$-Sylow subgroup $S_2$ of $U$.  Thus $S_2/\O_p(U)$ is a finite $p$-Sylow subgroup of $U/\O_p(U)$, so $U$ is virtually $p$-normal.

(ii) Let $U$ be an open compact subgroup of $G$ with $p$-Sylow subgroup $S$, let $K = \phi(U)$ and let $R = \phi(S)$.  Then $K$ is compact, so there is some open compact subgroup $V$ of $H$ such that $K \le V$.  Furthermore, $R$ is a $p$-Sylow subgroup of $K$, and the fact that $|V:\O_p(V)|_p$ is finite ensures that $\O_p(K)$ contains an open subgroup of $R$, since $\O_p(K) \ge \O_p(V) \cap K$.  In particular, $\Comm_K(R) = R$.  Consequently $\Comm_U(L)=U$, where $L = S\ker(\phi|_U)$.  Let $u \in U$.  Then $uSu^{-1}$ is a $p$-Sylow subgroup of $uLu^{-1}$; since $uLu^{-1}$ is commensurate to $L$, in fact $uSu^{-1} \cap L$ is open in $uSu^{-1}$ and has finite index in a $p$-Sylow subgroup of $L$, so by Sylow's theorem there is some $l \in L$ such that $uSu^{-1}$ is commensurate to $lSl^{-1}$.  Thus $U = \Comm_U(S)L = \Comm_U(S)\ker(\phi|_U)$.  But $G=\Comm_G(S)U$ by part (i), so $G = \Comm_G(S)\ker\phi$, that is, $\phi(G) = \phi(\Comm_G(S))$.

(iii) Let $U$ be an open compact subgroup of $G$.  Then $S \cap U$ is commensurate to a $p$-Sylow subgroup $T$ of $U$ by Lemma \ref{sylplcs}.  In turn $T$ contains $\O_p(U)$, so $\O_p(U) \le \Comm_G(T) = \Comm_G(S)$.

(iv) The fact that (a) implies (b) is just a special case of part (ii), and it is obvious that (b) implies (c).

Suppose that $\Comm_G(S)$ contains the dense normal subgroup $D$ of $G$.  Fix $x \in D \cap U$.  Given an open subgroup $T$ of $S$, let $X(T) := \{ u \in U \mid uxu^{-1} \in \N_U(T)\}$.  Then $X(T)$ is a closed subset of $U$ for every open subgroup $T$ of $S$, and $U$ is the countable union $\bigcup_{T \le_o S} X(T)$ by Corollary \ref{normprop}, since $uxu^{-1} \in \Comm_U(S \cap U)$ for all $u \in U$.  Thus by the Baire category theorem, there is some open subgroup $T$ of $S$ such that $X(T)$ has non-empty interior.  In other words, $X(T)$ contains $uV$ for some $u \in U$ and some open normal subgroup $V$ of $U$.  By conjugating $T$, we may assume $u =1$.  Since $|U:T|_p$ is finite, by replacing $V$ with a small enough open normal subgroup of $U$, we may assume $T \cap V$ is a $p$-Sylow subgroup of $V$.  In this case, $vxv^{-1}$ normalises $T \cap V$ for all $v \in V$, so by Sylow's theorem, $x$ normalises every $p$-Sylow subgroup of $V$.

Given an open normal subgroup $V$ of $U$, define $R(V)$ to be the intersection of the normalisers of all $p$-Sylow subgroups of $V$.  By the previous argument, we have $D \cap U \subseteq \bigcup_{V \unlhd_o U} R(V)$.  Moreover, by assumption $D \cap U$ is dense in $U$.  Since $\M(U)$ is open in $U$, there are therefore open normal subgroups $V_1,\dots,V_n$ of $U$ such that $\bigcup^n_{i=1} R(V_i)$ meets every coset of $\M(U)$ in $U$.  Let $V = \bigcap^n_{i=1} V_i$.  Then $R(V)$ contains $R(V_i)$ for all $i$ (since all $p$-Sylow subgroups of $V$ can be expressed as $V \cap T_i$ where $T_i$ is a $p$-Sylow subgroup of $V_i$), so $U = R(V)\M(U)$; in particular, $R(V)$ is not contained in any maximal open normal subgroup of $U$.  But $R(V)$ is a closed normal subgroup of $U$; since every proper closed normal subgroup of a profinite group is contained in a maximal open normal subgroup, we conclude that $R(V) = U$.  In particular, $R(V) \ge V$, so $V$ has a normal $p$-Sylow subgroup, proving (a).\end{proof}

It is easy to produce examples of profinite groups $G$ that are not virtually $p$-normal for some $p$, so that by Theorem \ref{pnormthm} (i), $G_{(p)}$ is a topologically residually finite \tdlc{} group with no open compact normal subgroups (indeed, one can arrange for $G_{(p)}$ to have no non-trivial compact normal subgroups).  For instance, let $G$ be the free profinite group (or free pro-$\pi$ group, where $\pi$ is a set of at least two primes) on two generators, and consider the $p$-localisation of $G$ (for some $p \in \pi$).

\

We define the modular function $\Delta$ on $G$ as:
\[\Delta:G \rightarrow \bQ^{\times}_{>0}; \; x \mapsto \frac{\mu(xU)}{\mu(U)}\]
where $\mu$ is a right-invariant Haar measure on $G$ and $U$ is any open compact subgroup of $G$.  It follows easily from the definitions that $\Delta(x) \equiv s(x)s(x^{-1})^{-1}$ is an identity for any \tdlc{} group.

\begin{proof}[Proof of Proposition \ref{goodmod}]Let $U$ be an open compact subgroup of $G$.  We can find an open subgroup $T$ of $S$ such that $T \le U$.  Then
\[ \Delta(x)_p = \frac{|xUx^{-1}:xUx^{-1} \cap U|_p}{|U:xUx^{-1} \cap U|_p} =  \frac{|xUx^{-1}:xTx^{-1} \cap T|_p}{|U:xTx^{-1} \cap T|_p}\]
\[= \frac{|xUx^{-1}:xTx^{-1}|_p|xTx^{-1}:xTx^{-1} \cap T|}{|U:T|_p|T:xTx^{-1} \cap T|} = \frac{|xTx^{-1}:xTx^{-1} \cap T|}{|T:xTx^{-1} \cap T|} = \Delta_{(p)}(x).\]
The validity of this calculation is ensured by the fact that $T$ and $xTx^{-1}$ are both commensurate to $p$-Sylow subgroups of $U$ and likewise  for $xUx^{-1}$, so the $p$-parts of all the relevant subgroup indices are finite.\end{proof}

\begin{proof}[Proof of Theorem \ref{locscale}](i) Let $U$ be tidy for $x$ in $G$ and let $W$ be a subgroup of $U$ that is tidy for $x$ in $G_{(p)}$.  Let $U_n = U \cap x^{-1}Ux \cap \dots \cap x^{-n}Ux^n$, and define $W_n$ similarly.  Then $|U:U_n| = s(x)^n$ and $|W:W_n| = s_{(p)}(x)^n$: this follows by induction from the fact that conjugation by $x$ preserves the set of subgroups that are tidy for $x$, together with the fact (\cite{Wil1} Lemma 10) that the intersection of tidy subgroups is tidy.  As $U_n$ contains $W_n$, we have the following:
\[ |U:U_n|_p \leq |U:W_n|_p = |U:W|_p|W:W_n| = |U:W|_p s_{(p)}(x)^n.\]
Now $W$ is commensurate to a local $p$-Sylow subgroup of $G$, so $|U:W|_p$ is finite.  Hence $(s(x)_p/s_{(p)}(x))^n$ is bounded as $n \rightarrow \infty$, so $s(x)_p \leq s_{(p)}(x)$.

(ii) By (\cite{Wil1}, Theorem 3), we have $s(xu) = s(x)$ for any $u \in U$; moreover $U$ remains tidy for $xu$.  Since $\Comm_G(S)$ is dense in $G$, by multiplying by a suitable element of $U$ we may assume $x \in \Comm_G(S)$.

We see that $S \cap U \cap xUx^{-1}$ is an open subgroup of a $p$-Sylow subgroup $R$ of $U \cap xUx^{-1}$.  In turn, $R$ is contained in a $p$-Sylow subgroup $T$ of $U$, and a $p$-Sylow subgroup $T^*$ of $xUx^{-1}$; by Sylow's theorem, $T^* = xuTu^{-1}x^{-1}$ for some $u \in U$.  Thus $xuT(xu)^{-1} \cap T = R$.  In particular, both $T$ and $xu$ are contained in $\Comm_G(S)$, and $T$ is an open compact subgroup of $G_{(p)}$.  Thus we have
\[ |xuT(xu)^{-1}:xuT(xu)^{-1} \cap T| = |T^*:R| = |xUx^{-1}:xUx^{-1} \cap U|_p.\]
In particular, $s_{(p)}(xu) \le s(xu)_p$.  We have already seen $s_{(p)}(xu) \ge s(xu)_p$, so equality holds.\end{proof}

\begin{rem}\label{coprem}(i) Theorem \ref{locscale} (i) should be compared with Proposition \ref{sgpscale}.  In particular, given a closed subgroup $H$ of $G$ such that $S \cap H$ is open in $S$ and $H$, then $s(x)_p \le s_{(p)}(x) \le s(x)$ for all $x \in \N_G(H)$, and if $\N_G(H)$ is open in $G$ then in fact we are forced to have $s_{(p)}(x) = s(x)_p$ for all $x \in \N_G(H)$.  Indeed, if $\N_G(H)$ is open in $G$ then $\Comm_G(S) = G$, so $G_{(p)}$ is simply a topological refinement of $G$.

(ii) Using the notation of Theorem \ref{locscale} (ii), let $x \in \Comm_G(S)$ and suppose $s(x)$ and $s(x^{-1})$ are both coprime to $p$.  We see in this case that $x$ is the product of two uniscalar elements of $G_{(p)}$, viz. $x = (xu)u^{-1}$, where $u$ is as in Theorem \ref{locscale} (ii).  Here $s_{(p)}(xu) = 1$ by Theorem \ref{locscale} (ii) (and thus $s_{(p)}((xu)^{-1})=1$ by Proposition \ref{goodmod}), while $u^{-1}$ is uniscalar in $G_{(p)}$ by Lemma \ref{wilsclem} and Corollary \ref{normprop}.  However, in general the product of two uniscalar elements of a \tdlc{} group need not be uniscalar, and indeed the function $s_{(p)}$ may be non-trivial even if the scale on $G$ only takes values coprime to $p$, as we will see in the next section.\end{rem}

Let $U$ be a profinite group.  Given a set of primes $\pi$, the \emph{$\pi$-core} $\O_{\pi}(U)$ is the unique largest pro-$\pi$ normal subgroup of $U$.  Note that by Sylow's theorem, $\O_p(U)$ is the intersection of all $p$-Sylow subgroups of $U$.  More generally, given a set $\mcP$ of sets of primes, one can define $\O_\mcP(U) := \prod_{\pi \in \mcP} \O_\pi(U)$.  The most important example is the \emph{pro-Fitting subgroup} $\F(U) := \prod_{p \in \bP} \O_p(U)$.  The pro-Fitting subgroup also has other characterisations, for instance it is the unique largest pronilpotent subnormal subgroup of $U$.  The following observation means that these are also useful definitions in the context of \tdlc{} groups:

\begin{prop}\label{corecomm}Let $G$ be a \tdlc{} group and let $U$ and $V$ be open compact subgroups of $G$.  Let $\mcP$ be a set of pairwise disjoint sets of primes.  Then $\O_\mcP(U)$ is commensurate to $\O_\mcP(V)$.  In particular $G = \Comm_G(\O_\mcP(U))$.\end{prop}

\begin{proof}Say a profinite group $K$ is a pro-$\mcP$ group if $K$ is the product of normal pro-$\pi$ subgroups as $\pi$ ranges over $\mcP$.  Note that pro-$\mcP$ groups are characterised by the following property: they are pro-$\bigcup \mcP$ groups such that given distinct (and hence disjoint) elements $\pi_1,\pi_2 \in \mcP$ and pro-$\pi_i$ subgroups $L_i$ for $i=1,2$, then $[L_1,L_2]=1$.  This property is inherited by subgroups.  Note also that $\O_\mcP(U)$ is the unique largest normal pro-$\mcP$ subgroup of $U$.  Thus if $V$ is a normal subgroup of $U$, then $\O_\mcP(U) \cap V$ is a normal pro-$\mcP$ subgroup of $V$ and $\O_\mcP(V)$ is a normal pro-$\pi$ subgroup of $U$, so $\O_\mcP(U) \cap V = \O_\mcP(V)$.  In particular, $\O_\mcP(V)$ has finite index in $\O_\mcP(U)$.

In the general case, the following situation occurs: there is an open subgroup $W_1$ of $U \cap V$ that is normal in $U$, and then an open subgroup $W_2$ of $W_1$ that is normal in $V$.  Thus we have $\O_\mcP(U) \sim \O_\mcP(W_1) \sim \O_\mcP(W_2) \sim \O_\mcP(V)$, where $A \sim B$ means that $A$ and $B$ are commensurate.\end{proof}

\section{Example: Groups acting on trees with prescribed local action}\label{egsec}

In this section we will describe some features of local Sylow theory for a particular class of \tdlc{} groups (essentially, those studied in \cite{BM}) to illustrate the results of previous sections and give counterexamples to some natural generalisations.  Unlike in \cite{BM}, it is useful for our purposes to allow the local action to be intransitive.

The following notation will be assumed for the remainder of this section.

\begin{defn}Let $k \ge 3$ be a natural number, let $T$ be the regular $k$-valent tree and let ${\bf k} = \{1,\dots,k\}$.  Let $\Aut(T)$ be the automorphism group of $T$ as a graph.  Write $E(T)$ for the set of undirected edges of $T$.  Let $c: E(T) \rightarrow {\bf k}$ be a proper edge-${\bf k}$-colouring of $T$, that is, for every vertex $v$ of $T$ there is a right inverse of $c$ of the form $c^*_v: {\bf k} \rightarrow E(T)$ such that the image of $c^*_v$ is the set of edges incident with $v$.  We will also mark a vertex $\emptyset \in T$, and an arbitrary vertex $v$ can then be referred to uniquely by the sequence of colours on the path from $\emptyset$ to $v$.  For $g \in \Aut(T)$ and $v \in T$, define $g@v := cgc^*_v$, and say $g$ acts at $v$ as $g@v$.    Now given a subgroup $F$ of $\Sym({\bf k})$, define $U(F)$ to be the set of $g \in \Aut(T)$ such that $g@v \in F$ for all $v \in T$.  Then $U(F)$ is a subgroup of $\Aut(T)$; moreover $U(F)$ is closed, since the elements $g \in \Aut(T)$ for which $g@v$ is contained in $F$ form a closed subset for any $v \in T$.  The full automorphism group of $T$ is just $U(\Sym({\bf k}))$.  Let $\partial T$ be the space of ends of $T$.  Given $v,w \in T \cup \partial T$, we use interval notation for the unique geodesic from $v$ to $w$; we will refer to geodesics with one end in $\partial T$ as \emph{rays} and those with both ends in $\partial T$ as \emph{lines}.  Notice that for all $v \in T \cup \partial T$, the image $gv$ is determined uniquely by $g\emptyset$ together with the sequence of colours $g@w(c(ww'))$ as $ww'$ runs along all edges on the geodesic $[\emptyset,v)$.  In particular, writing $1_{\bf k}$ for the trivial subgroup of $\Sym({\bf k})$, then $U(1_{\bf k})$ acts regularly on the vertices of $T$; as a group, $U(1_{\bf k})$ is the free product of the $k$ involutions induced by reversing an edge incident with $\emptyset$, and $T$ can be seen as the Cayley graph of $U(1_{\bf k})$.  We will refer to elements of $U(1_{\bf k})$ as \emph{colour-preserving}.

Recall that there are two kinds of automorphisms of $T$: those that preserve a finite subtree (\emph{elliptic} automorphisms), and those that preserve a unique line $\Pi$, such that the action restricted to $\Pi$ is a translation (\emph{hyperbolic} automorphisms).  (Some authors exclude \emph{inversions}, that is automorphisms that reverse an edge, from the class of elliptic automorphisms; for the present discussion, however, the distinction serves no purpose.)

Let $\bN_0$ be the set of natural numbers including $0$.\end{defn}

The following is clear from the definitions.

\begin{lem}\label{treesyl}Let $G = U(F)$ and let $F(p)$ be a $p$-Sylow subgroup of $F$.  Then $S = U(F(p))_\emptyset$ is a local $p$-Sylow subgroup of $G$, and $\Comm_G(S)$ consists of all $g \in G$ such that $g@v \in \N_F(F_p)$ for all but finitely many $v \in T$.\end{lem}

From now on we will assume $F(p)$ is a fixed $p$-Sylow subgroup of $F$, and that the $p$-localisation of $G$ is one arising from Lemma \ref{treesyl}.

There is a simple formula for the scale on $U(F)$, so that determining the values taken by the scale on $U(F)$ reduces to a consideration of the suborbits of $F$ (that is, orbits of point stabilisers of $F$) in its action on ${\bf k}$.

\begin{prop}\label{tidysgp}Let $x \in G = U(F)$ and let $s$ be the scale on $G$.
\begin{enumerate}[(i)]
\item If $x$ is elliptic then $s(x)=1$.
\item If $x$ is hyperbolic, let $\Pi$ be the line preserved by $x$.  Let $v \in \Pi$ and let $U = G_v \cap G_{xv}$, in other words $U$ is the orientation-preserving stabiliser of $[v,xv]$.  Then $U$ is tidy for $x$, and
\[ s(x) = |Ux^{-1}v| = \prod^n_{i=1} |F_{c_{i-1}}c_i|,\]
where $c_o$ is the colour of the edge incident with $v$ in the direction of $xv$, and $(c_1,\dots,c_n)$ is the sequence of colours formed by the edges in $[v,x^{-1}v]$.\end{enumerate}\end{prop}

\begin{proof}(i) This is clear since $x$ is contained in an open compact subgroup, namely either a vertex stabiliser or an undirected edge stabiliser.

(ii) Recall the notation of Definition \ref{tidydef}.  Let $\omega_+$ and $\omega_-$ be the ends of $\Pi$, chosen so that $xv \in [v,\omega_+)$.  We see that $U_+$ is the stabiliser of the ray $[v,\omega_+)$ and $U_-$ is the stabiliser of the ray $[xv,\omega_-)$.  Since the actions of $U$ on the connected components of $T \setminus \{v,xv\}$ are independent of each other, we have $U=U_+U_-$.  Meanwhile $U_{++} = G_{\omega_+}$ and $U_{--} = G_{\omega_-}$, so both $U_{++}$ and $U_{--}$ are closed.  Hence $U$ is tidy for $x$.

Thus the scale of $x$ is given by $|xUx^{-1}:xUx^{-1} \cap U| = |xUx^{-1}v| = |Ux^{-1}v|$.  Let the path $[v,x^{-1}v]$ have vertices $v =w_0,w_1,\dots,w_n = x^{-1}v$.  Then for $1 \le i \le n$, the orbit of $w_i$ under $U_{w_{i-1}}$ consists of those neighbours of $w_{i-1}$ connected to $w_{i-1}$ by an edge with colour in $F_{c_{i-1}}c_i$, since the edge of colour $c_{i-1}$ incident with $w_{i-1}$ is fixed.  The given formula for the scale follows by repeated application of the orbit-stabiliser theorem.\end{proof}

We thus obtain some properties of the set of values taken by the scale.

\begin{cor}\label{scaleval}Let $s$ be the scale on $U(F)$.
\begin{enumerate}[(i)]
\item Let $i$ and $j$ be distinct elements of ${\bf k}$.  Then there exists $x \in U(1_{\bf k})$ such that $s(x) = |F_i j||F_j i|$.
\item Let $i$ and $j$ be elements of ${\bf k}$ in the same $F$-orbit.  Then there exists $x \in U(F)$ such that $s(x) = |F_i j|$.
\item Suppose that $F$ is $2$-transitive.  Then $s(U(F))=\{ (k-1)^n \mid n \in \bN_0 \}$.
\end{enumerate}\end{cor}

\begin{proof}(i) Let $\Pi$ be a line whose edges are alternately coloured $i$ and $j$.  Then there is a colour-preserving automorphism $x$ (unique up to taking its inverse) such that $x\Pi = \Pi$ and $x|_\Pi$ is a translation of distance $2$.  The scale value follows from the formula.

(ii) We may assume $i \not = j$, since if $i=j$ then $|F_i j| = 1 = s(1)$.  Let $\tau \in F$ be such that $\tau j = i$, and suppose $\tau^l j = j$ for some $l$.  Let $\Pi$ be a line whose edge colours cycle through $(j, i, \tau^2 j, \dots, \tau^{l-1}j)$.  Then there exists $x \in U(F)$ such that $x\Pi = \Pi$ and $x|_\Pi$ is a translation of distance $1$.  We choose $x$ and $v \in \Pi$ such that the edge from $v$ to $xv$ has colour $i$ and the edge from $x^{-1}v$ to $v$ has colour $j$, giving the required expression for the scale.

(iii) All suborbits of $F$ have length $k-1$, so every value taken by the scale is a power of $k-1$, and the value $k-1$ itself occurs by (ii).\end{proof}

The structure of $U(F)_{(p)}$ is more difficult to describe geometrically than $U(F)$ itself, since $U(F)_{(p)}$ is usually not compactly generated and may not even be $\sigma$-compact.  Notice however that $U(F)_{(p)}$ has an open subgroup topologically isomorphic to $U(F(p))$, so in light of Lemma \ref{wilsclem} we see that the scales of elements of $U(F(p))$ (which we can determine easily using Proposition \ref{tidysgp}) form a subset of the values of the scale on $U(F)_{(p)}$.  This is already enough to show that the $p$-part of the scale on $G$ does not determine the scale on $G_{(p)}$.

\begin{prop}\label{symscaleval}Let $F = \Sym({\bf k})$ and let $p$ be a prime.  Let $s$ be the scale on $U(F) = \Aut(T)$ and let $s'$ be the scale on $U(F(p))$.  Let $\mcS$ be the set of those integers $e$ such that $p^e$ occurs as the $p$-part of $s(x)$ for $x \in U(F)$.  Let $\mcT$ be the set of those integers $e$ such that $s'(y)=p^e$ for some $y \in U(F(p))$.
\begin{enumerate}[(i)]
\item If $k \le p$ then $\mcT = \{0\}$.
\item If $p > 2$ and $k = 2p$ then $\mcT = 2\bN_0$.
\item If $p > 3$ and $k = rp$ for $3 \le r < p$ then $\mcT = \bN_0 \setminus \{1\}$.
\item In all other cases $\mcT = \bN_0$.
\item For all values of $k$ ($k \ge 3$) we have $\mcS = e\bN_0$ where $p^e$ is the largest power of $p$ dividing $k-1$.  In particular, if $k \not\equiv 1 \mod p$ then $\mcS = \{0\}$.
\end{enumerate}\end{prop}

\begin{proof}If $k < p$ then $F(p)$ is trivial; if $k=p$ then $F(p)$ acts regularly.  In either case, there are no suborbits of length greater than $1$, so $\mcT = \{0\}$ by Proposition \ref{tidysgp}.

For (ii) and (iii), $F(p)$ is elementary abelian of order $p^{k/p}$, acting independently on $k/p$ orbits of size $p$.  Let $y \in U(F(p))$ be hyperbolic and let $(c_0,\dots,c_n)$ be as in Proposition \ref{tidysgp} (ii).  We see that $s'(y) = p^l$ where $l$ is the number of pairs $(c_i,c_{i+1})$ that lie in different $F(p)$-orbits.  However, we have the further constraint that $c_0$ is in the same $F(p)$-orbit as $c_n$.  If there are exactly two $F(p)$-orbits, we can obtain any non-negative even number for $l$, but no odd numbers.  If there are three or more $F(p)$-orbits, then $l$ can be any non-negative integer except $1$.

For (iv) it suffices to show $p \in \mcT$.  We are left with cases (a) and (b) below:

\emph{(a) $p < k < p^2$ and $p$ does not divide $k$.}

In this case there exist $i,j \in {\bf k}$ such that $i$ is fixed by $F(p)$ and $|F(p)j|=p$.  By Corollary \ref{scaleval} (i) there is therefore $y \in U(1_{\bf k})$ such that $s'(y) = p$.

\emph{(b) $k \ge p^2$.}

In this case, there is an $F(p)$-invariant equivalence relation on ${ \bf k}$ with each equivalence class lying in a single $F(p)$-orbit, and such that at least one of the equivalence classes $X$ has size $p^2$.  Let $i \in X$.  Then $F(p)_i$ acts on $X$ intransitively, with $p$ orbits of size $1$ and $p-1$ orbits of size $p$.  Let $j$ be in an $F(p)_i$-orbit of size $p$.  We obtain an element $y \in U(F(p))$ such that $s'(y)=p$ as in Corollary \ref{scaleval} (ii).

Part (v) is immediate from Corollary \ref{scaleval} (iii).\end{proof}

\begin{cor}Let $p$ be a prime and let $k$ be an integer such that $k > p$ and $p$ does not divide $k-1$.  Let $G = U(\Sym(k))$.  Then all values of the scale on $G$ are coprime to $p$, but $G_{(p)}$ is not uniscalar.\end{cor}

\section{Locally virtually prosoluble groups}

Just as Sylow's theorem generalises to profinite groups, Hall's theorem on Sylow bases generalises to prosoluble groups, that is, groups that are inverse limits of finite soluble groups.

\begin{defn}Let $G$ be a prosoluble group (that is, an inverse limit of finite soluble groups).  A \emph{Sylow basis} of $G$ is a set $\Sigma$ of subgroups containing one $p$-Sylow subgroup for each prime $p$ such that the elements of $\Sigma$ are pairwise permutable.  A \emph{local Sylow basis} of a locally virtually prosoluble group is a Sylow basis of some open prosoluble subgroup.\end{defn}

\begin{thm}[part of Hall's theorem for prosoluble groups: see \cite{Bol} Theorems 13 and 14]Let $G$ be a prosoluble group.  Then $G$ has a Sylow basis, and all Sylow bases of $G$ are conjugate in $G$.\end{thm}

\begin{cor}\label{sylslcs}Given a \tdlc{} group $G$ that is locally virtually prosoluble, write $\Syls(G)$ for the set of all Sylow bases of prosoluble open compact subgroups of $G$.  Then $\Syls$ is a locally conjugate structure on the class of locally virtually prosoluble \tdlc{} groups.\end{cor}

\begin{proof}Let $G$ be a prosoluble group.  Given an open subgroup $H$ of $G$ and a Sylow basis $\Xi$ of $H$, then $\Xi \le_o \Sigma$.  Conversely, given any Sylow basis $\Sigma$ of $G$, then $\Sigma \cap K$ is a Sylow basis of $K$, where $K$ is the core of $H$ in $G$, and so $\Sigma$ is commensurate to a Sylow basis of $H$.  Thus setting $\mcL(G)$ to be the set of Sylow bases of a prosoluble group $G$ gives a locally conjugate structure on prosoluble groups.  It follows from Lemma \ref{lcslem} that $\Syls$ is a locally conjugate structure on all locally virtually prosoluble \tdlc{} groups.\end{proof}

We can now proceed as before, considering $\Comm_G(\Sigma)$ where $\Sigma$ is a local Sylow basis in place of $\Comm_G(S)$ for $S$ a local Sylow subgroup.  We do not attempt to define a new topology for $\Comm_G(\Sigma)$, but we can nevertheless adapt parts of Theorems \ref{commdense} and \ref{pnormthm} to the present context.

The following analogue of Theorem \ref{commdense} is immediate:

\begin{thm}\label{commdense2}Let $G$ be a \tdlc{} group that is locally virtually prosoluble and let $\Sigma$ be a local Sylow basis of $G$.  Then $\Comm_G(\Sigma)$ is a dense subgroup of $G$, and the $G$-conjugacy class of $\Comm_G(\Sigma)$ does not depend on the choice of $\Sigma$.\end{thm}

Parts (ii) and (iv) of Theorem \ref{pnormthm} also have analogues in the present context.

\begin{thm}\label{pnormthm2}Let $G$ be a \tdlc{} group that is locally virtually prosoluble and let $\Sigma$ be a local Sylow basis of $G$.
\begin{enumerate}[(i)]
\item Let $H$ be a \tdlc{} group that is locally virtually pronilpotent.  Let $\phi: G \rightarrow H$ be a continuous homomorphism.  Then $\phi(\Comm_G(\Sigma)) = \phi(G)$.
\item Suppose that $G$ is first-countable and that there is an open compact subgroup $U$ of $G$ that has only finitely many maximal open normal subgroups.  Then the following are equivalent:
\begin{enumerate}[(a)]
\item $G$ is locally virtually pronilpotent;
\item $\Comm_G(\Sigma)=G$;
\item $\Comm_G(\Sigma)$ contains a dense normal subgroup of $G$.\end{enumerate}\end{enumerate}\end{thm}

\begin{lem}\label{hallfrat}Let $U$ be a prosoluble group and let $\Sigma$ be a Sylow basis for $U$.  Let $K$ be a closed normal subgroup of $U$ such that $U/K$ is pronilpotent.  Then $U = \N_U(\Sigma)K$.\end{lem}

\begin{proof}Given a prime $p$, write $S_p$ for the $p$-Sylow subgroup of $U$ contained in $\Sigma$.  Since $\N_U(\Sigma)$ is a closed subgroup of $U$, it suffices to prove the result in the case that $K$ is open in $U$.  Given an open normal subgroup $N$ of $G$, let $\Sigma_N = \{SN \mid S \in \Sigma\}$.  Then by a standard compactness argument, to show $U = \N_U(\Sigma)K$ it suffices to show $U = \N_U(\Sigma_N)K$ for every open normal subgroup $N$ of $U$.  This clearly reduces to proving the finite case of the lemma, which is a classical result.  Indeed, it is a theorem of P. Hall (\cite{Hal}) that in a finite soluble group $U$, the normaliser of a Sylow basis covers the central chief factors of $U$ (and only these chief factors), so in particular, the normaliser of a Sylow basis covers every nilpotent image.\end{proof}

\begin{proof}[Proof of Theorem \ref{pnormthm2}](i) Note that the choice of $\Sigma$ is not important, because all local Sylow basis commensurators in $G$ are conjugate. Let $U$ be an open prosoluble subgroup of $G$.  Let $M = \phi(U)$ and let $K = \ker\phi \cap U$.  The hypothesis on $H$ ensures that $M$ is virtually pronilpotent.  By replacing $U$ with a suitable open subgroup, we may assume $M$ is actually pronilpotent.  Let $\Sigma$ be a Sylow basis for $U$.  By Lemma \ref{hallfrat}, we have $U = \N_U(\Sigma)K$.  Since $G=\Comm_G(\Sigma)U$ by Theorem \ref{commdense2}, we have $G = \Comm_G(\Sigma)\ker\phi$, that is, $\phi(G) = \phi(\Comm_G(\Sigma))$.

(ii) The fact that (a) implies (b) is just a special case of part (i), and it is obvious that (b) implies (c).

Suppose that $\Comm_G(\Sigma)$ contains the dense normal subgroup $D$ of $G$.  Fix $x \in D \cap U$.  Given $i \in \bN$, let $U_i$ consist of those elements $u \in U$ such that $uxu^{-1}$ normalises the local $p$-Sylow subgroup contained in $\Sigma$ for all $p \ge i$.  Since $uxu^{-1} \in \Comm_G(\Sigma)$ for all $u \in U$, we have $\bigcup_{i \in \bN}U_i = U$; clearly each $U_i$ is closed.  Thus by the Baire category theorem, there is some $i$ such that $U_i$ has non-empty interior and hence contains a coset of an open normal subgroup of $U$.  By the same argument as used in the proof of Theorem \ref{pnormthm} (iv), we conclude that there is an open normal subgroup $V$ of $U$ such that $x$ normalises every $p$-Sylow subgroup of $V$ for all $p \ge i$.  By Theorem \ref{pnormthm} (iv), we see that $G$ is locally virtually $p$-normal for every prime $p$, since $\Comm_G(S)$ contains $D$ for all $S \in \Sigma$.  Thus by intersecting $V$ with open normal subgroups of $U$ that are $p$-normal, one for each $p < i$, we may assume that $V$ is $p$-normal for all $p < i$, so that $x$ normalises every Sylow subgroup of $V$.

Given an open normal subgroup $V$ of $U$, define $R(V)$ to be the intersection of the normalisers of all Sylow subgroups of $V$.  By the previous argument, we have $D \cap U \subseteq \bigcup_{V \unlhd_o U} R(V)$.  As in the proof of Theorem \ref{pnormthm} (iv), we obtain an open normal subgroup $V$ of $U$ such that $R(V) = U$.  In particular, every Sylow subgroup of $V$ is normal, which means that $V$ is pronilpotent, proving (a).\end{proof}

Given a Sylow basis $\Sigma$, one can define an aggregate scale for elements of $\Comm_G(\Sigma)$ by taking the product of the scales in each $p$-localisation:
\[ s_*(x) := \prod_{p \in \bP} s_{(p)}(x).\]
(Here $\bP$ is the set of all prime numbers, and the localisation is taken at the local $p$-Sylow subgroup $S$ contained in $\Sigma$, so that $\Comm_G(\Sigma) \le \Comm_G(S)$.)

We obtain an analogue of Theorem \ref{locscale}.

\begin{thm}\label{sysscale}Let $G$ be a locally virtually prosoluble \tdlc{} group and let $\Sigma$ be a local Sylow basis of $G$.  Let $x \in G$.
\begin{enumerate}[(i)]
\item If $x \in \Comm_G(\Sigma)$, then $s_*(x)$ is a finite multiple of $s(x)$.
\item Let $U$ be a tidy subgroup for $x$ in $G$.  Then there exists $u \in U$ such that $xu$ commensurates $\Sigma$ and such that $s(x)=s(xu) = s_*(xu)$.\end{enumerate}\end{thm}

\begin{proof}(i) From the definition of commensuration, we see that $x$ normalises all but finitely many elements of $\Sigma$, so $s_{(p)}(x) = 1$ for all but finitely many $p$; hence $s_*(x)$ is finite.  The fact that $s_*(x)$ is a multiple of $s(x)$ follows immediately from Theorem \ref{locscale} (i).

(ii) The equation $s(x)=s(xu)$ is automatic and we may assume $x \in \Comm_G(\Sigma)$, for the reasons given in the proof of Theorem \ref{locscale} (ii).

We see that $\Sigma \cap U \cap xUx^{-1} \le_o \Sigma_1$, where $\Sigma_1$ is a Sylow basis of $U \cap xUx^{-1}$.  In turn, $\Sigma_1 \le_o \Sigma_2$ and $\Sigma_1 \le_o \Sigma_3$ for Sylow bases $\Sigma_2$ and $\Sigma_3$ of $U$ and $xUx^{-1}$ respectively; by Hall's theorem, $\Sigma_3 = xu\Sigma_2 u^{-1}x^{-1}$ for some $u \in U$.  Thus for each $p \in \bP$, the $p$-Sylow subgroup $R$ in $\Sigma_1$ arises as $xuT(xu)^{-1} \cap T$ such that $T$ is the $p$-Sylow subgroup in $\Sigma_2$ and $xuT(xu)^{-1}$ is the $p$-Sylow subgroup in $\Sigma_3$.  In particular, both $T$ and $xu$ are contained in $\Comm_G(S)$, and $T$ is an open compact subgroup of $G_{(p)}$.  Thus we have
\[ |xuT(xu)^{-1}:xuT(xu)^{-1} \cap T| = |T^*:R| = |xUx^{-1}:xUx^{-1} \cap U|_p.\]
In particular, $s_{(p)}(xu) = s(xu)_p$ as in Theorem \ref{locscale} (ii).  This argument holds for every $p \in \bP$, so $s(xu) = s_*(xu)$.\end{proof}

\begin{rem}We make an observation analogous to Remark \ref{coprem} (ii): given $x \in \Comm_G(\Sigma)$ such that $x$ is uniscalar in $G$, then $x = yz$ where $y,z \in \Comm_G(\Sigma)$ and $s_*(y)=s_*(y^{-1})=s_*(z)=s_*(z^{-1})=1$.\end{rem}

One can define an `aggregate modular function' in the same way as the aggregate scale, but it turns out to be just the original modular function of $G$ restricted to $\Comm_G(\Sigma)$.  Indeed, this is clear from Proposition \ref{goodmod}: for all $x \in \Comm_G(\Sigma)$, we have
\[ \frac{s_*(x)}{s_*(x^{-1})} = \prod_{p \in \bP} \frac{s_{(p)}(x)}{s_{(p)}(x^{-1})} = \prod_{p \in \bP} \Delta_{(p)}(x) = \prod_{p \in \bP} \Delta(x)_p = \Delta(x).\]

\section{Acknowledgements}

My thanks go to Pierre-Emmanuel Caprace and George Willis for helpful discussions with them.

\end{document}